\documentclass[12pt,a4paper]{amsart}
\usepackage{amssymb, amstext, amscd, amsmath, color}

\usepackage{url}

\usepackage[dvips]{graphicx}

\usepackage{amsfonts}
\usepackage{amsthm}
\usepackage{amsfonts}
\usepackage{array}
\usepackage{epsfig}
\usepackage{eucal}
\usepackage{latexsym}
\usepackage{mathrsfs}
\usepackage{textcomp}
\usepackage{verbatim}
\usepackage[utf8]{inputenc}
\usepackage{tikz}

\newtheorem{thm}{Theorem}[section]

\newtheorem{con}[thm]{Conjecture}

\newtheorem{defn}[thm]{Definition}

\newtheorem{lem}[thm]{Lemma}

\numberwithin{equation}{section}

\newcommand{\bN}{{\mathbb{N}}}

\newcommand{\bR}{{\mathbb{R}}}

  \newcommand{\M}{{\mathcal{M}}}
  \newcommand{\N}{{\mathcal{N}}}

  \newcommand{\R}{{\mathcal{R}}}

\begin{document}

\title{An Inductive Construction of $(2,1)$-tight Graphs}
\author[A. Nixon]{A. Nixon}
\address{Heilbronn institute for mathematical research, University of Bristol\\
Bristol, BS8 1TW, U.K. }
\email{tony.nixon@bristol.ac.uk}
\thanks{First author supported by  EPSRC grant EP/P503825/1}
\author[J.C. Owen]{J.C. Owen}
\address{D-Cubed, Siemens PLM Software, Park House\\
Castle Park, Cambridge, U.K.}
\email{owen.john.ext@siemens.com}
%\author[S.C. Power]{S.C.  Power}
%\thanks{Second author partially supported by an SERC grant.}
%\address{Dept.\ Math.\ Stats.\\ Lancaster University\\
%Lancaster LA1 4YF \\U.K. }
%\email{s.power@lancaster.ac.uk}

%\begin{center}


\date{}

\begin{abstract}
The graphs $G=(V,E)$ with $|E|=2|V|-\ell$ that satisfy $|E'|\leq 2|V'|-\ell$ for any subgraph $G'=(V',E')$ (and for $\ell=1,2,3$) are the $(2,\ell)$-tight graphs. The Henneberg--Laman theorem characterises $(2,3)$-tight graphs 
inductively in terms of two simple moves, known as the Henneberg moves. Recently this has been extended, via the addition of a graph 
extension move, to the case of $(2,2)$-tight simple graphs. Here an alternative characterisation is provided by means of vertex-to-$K_4$ and 
edge-to-$K_3$ moves, and this is extended to the $(2,1)$-tight simple graphs by addition of an edge joining move.
\end{abstract}
\maketitle

%\bigskip

% \input{type1include}

\section{Introduction}
\label{sec31}

The purpose of this paper is to prove an inductive characterisation of simple $(2,1)$-tight graphs.

\begin{defn}[Lee and Streinu \cite{L&S}]
Let $k,\ell \in \bN$ and $\ell \leq 2k$. A graph $G=(V,E)$ is {\rm $(k,\ell)$-sparse} if for every subgraph $G'=(V',E')$, 
$|E'|\leq k|V'|-\ell$ (where if $\ell=2k$ the inequality only applies if $|V'|\geq k$). $G$ is {\rm $(k,\ell)$-tight} if $G$ is 
$(k,\ell)$-sparse and $|E|=k|V|-\ell$.
\end{defn}

In our notation a graph allows parallel edges and loops, whereas a simple graph allows neither.

The classes of $(2,\ell)$-tight simple graphs play an important role in the theory of $2$-dimensional bar-joint frameworks (see, for example, 
\cite{A&R} and 
\cite{GSS} for the general theory).  When $l=3$ these graphs correspond to generic frameworks that are minimally rigid when joints corresponding to the
vertices are constrained to lie on a plane (since any framework on a plane has three independent rigid-body motions) \cite{Lam}. When $l=2$ these graphs correspond to generic frameworks
which are minimally rigid when the joints are constrained to lie on the surface of a cylinder (since this surface allows two independent rigid-body motions) \cite{NOP}. When $l=1$
we expect the graphs to correspond to frameworks that are rigid when the joints are constrained to a surface which admits one independent rigid-body motion. These surfaces include
linearly swept  surfaces (such as an elliptical cylinder or any ruled surface with parallel rulings) and spun surfaces (such as a circular cone, torus or any 
surface formed by rotating a smooth curve). These surfaces are important in engineering since they are easily manufactured using the processes of extrusion and turning.

The characterisation of generic framework rigidity typically involves two distinct steps - an inductive construction of the relevant class of graphs and then a proof 
that the construction steps preserve the required rigidity properties.

The classical result of Henneberg \cite{Hen} characterises the class of $(2,3)$-tight graphs by recursive 
operations. Combining this 
with a result of Lovasz and Yemini \cite{L&Y}, extended by Recski \cite{Rec}, leads to:

\begin{thm}\label{thm1}[Henneberg \cite{Hen}, Lovasz and Yemini \cite{L&Y}, Recski \cite{Rec}]
For a graph $G=(V,E)$ the following are equivalent:
\begin{enumerate}
\item $G$ is $(2,3)$-tight,
\item $G$ is derivable from $K_{2}$ by the Henneberg 1 and Henneberg 2 moves,
\item for any edge $e\in E(K_{|V(G)|})$, $G\cup \{e\}$ is the edge-disjoint union of two spanning trees.
\end{enumerate}
\end{thm}

 Laman \cite{Lam} then characterised generic minimal rigidity on the plane by showing that the Henneberg 1 and Henneberg 2 moves preserve this property on the plane.

Nixon, Owen and Power \cite{NOP} obtained a characterisation of simple $(2,2)$-tight graphs, Theorem \ref{thm2}.

\begin{thm}\label{thm2}[Nixon, Owen and Power \cite{NOP}, Nash-Williams \cite{N-W}]
For a simple graph $G=(V,E)$ the following are equivalent:
\begin{enumerate}
\item $G$ is $(2,2)$-tight,
\item $G$ is derivable from $K_{4}$ by the Henneberg 1, Henneberg 2 and graph extension moves,
\item $G$ is the edge-disjoint union of two spanning trees.
\end{enumerate}
\end{thm}

In this characterisation a graph extension move replaces a vertex in the graph by an arbitrary $(2,2)$-tight graph which thereby becomes a $(2,2)$-tight 
subgraph in the extended graph.   \cite{NOP} also characterised generic minimal rigidity on the cylinder by showing that the Henneberg 1, Henneberg 2 and graph extension moves preserve this property on the cylinder.

Our main result is the following inductive construction of $(2,1)$-tight simple graphs.
By $K_{5}\setminus e$ we mean the graph formed from the complete graph on $5$ vertices by removing an edge, and by 
$K_{4}\sqcup K_{4}$ we mean the graph formed by taking two copies of $K_{4}$ that intersect in a copy of the complete graph $K_2$. The construction operations are defined at the start of Section \ref{tightchar}.

\begin{thm}\label{21theorem}
 For a simple graph $G$ the following are equivalent:
\begin{enumerate}
 \item $G$ is $(2,1)$-tight,
\item $G$ can be derived from $K_{5}\setminus e$ or $K_{4}\sqcup K_{4}$ by the Henneberg 1, Henneberg 2, vertex-to-$K_4$, edge joining and edge-to-$K_3$ moves,
\item $G$ is the edge-disjoint union of a spanning tree $T$ and a spanning subgraph $P$ in which every connected component contains exactly one cycle.
\end{enumerate}
\end{thm}

We expect that each of the construction moves in (2) of this theorem also preserves minimal generic rigidity on surfaces which admit one rigid body motion.  We present this as a conjecture for subsequent investigation.

As a by-product of our arguments we also show the following result giving an alternative inductive construction of $(2,2)$-tight graphs. The construction should be easier to apply since we only insert prescribed small subgraphs rather than an arbitrary graph in the class.

\begin{thm}\label{22refine}
For a simple graph $G=(V,E)$ the following are equivalent:
\begin{enumerate}
\item $G$ is $(2,2)$-tight,
\item $G$ is derivable from $K_{4}$ by the Henneberg 1, Henneberg 2, vertex-to-$K_4$ and edge-to-$K_3$ moves.
\end{enumerate}
\end{thm}

The main difficulty in proving theorem \ref{21theorem} is the requirement that the inductive construction involves only simple graphs. This requirement arises because we are 
interested in frameworks in which the distance between a pair of joints is the usual distance measured as a straight line in $3$-space. Minimal rigidity then clearly requires that
two vertices are joined by at most edge. Whitely \cite{Whi5} has considered frameworks embedded on surfaces in which the distance between a pair of joints is a geodesic distance over the surface. In this case a pair of vertices may be separated by more than one distinct geodesic distance and the class of graphs may be extended to include multiple edges between a pair of vertices. Similarly periodic frameworks \cite{B&S}, \cite{M&T}, \cite{Ros} on the plane may include edges connecting between different cells and result in graphs with multiple edges.

We note that for the case of $(k,\ell)$-tight graphs (permitting parallel edges and loops) there are elegant recursive constructions requiring Henneberg type operations only \cite{Fr&S}, \cite{F&S}.

A further motivation for our work is the hope that understanding the recursive constructions for $(2,\ell)$-tight graphs of the 
various types will provide insight into $(3,6)$-tight graphs. 
These are the graphs relevant to major open problems in $3$-dimensional rigidity theory  \cite{GSS}, \cite{T&W}, \cite{Whi4}.
Note that these graphs are necessarily simple and are outside the matroidal range.
Indeed for $\ell<6$ adding \emph{any} $\ell-3$ edges to a $(3,6)$-tight graph results in a graph with a decomposition into three edge 
disjoint spanning trees but for $l=6$ it does not, see \cite{Haa}.

From our main theorems one can quickly derive sparsity variants. That is, characterisations of $(2,\ell)$-sparsity in terms of 
recursive operations. If Conjecture \ref{tjcon} is true then this has applications in computer aided design \cite{Owen} where the 
emphasis is on establishing whether a system of constraint equations admits a matrix with linearly independent rows.

The paper is organised as follows. Section \ref{tightchar} defines the recursive moves we will consider. The key difficulty is the 
construction theory of Section \ref{sec2*}, in which we discuss the sufficiency of the moves. The main step is Lemma \ref{jco4}. 
Here we take a seemingly mild requirement that each edge in a copy of $K_3$ is in at least two copies of $K_3$ or is in a separate 
$(2,1)$-tight subgraph. This leads to the strong conclusion that every copy of $K_3$ is in a copy of $K_4$. This convenient property 
is used to derive the key implication in the proofs of Theorems \ref{21theorem} and \ref{22refine}. Finally Section \ref{21apps} 
discusses rigidity theory and potential applications of our results therein.

We would like to thank Stephen Power for some helpful discussions and the anonymous referees for a number of helpful comments.

\section{Simple $(2,\ell)$-tight Graphs}
\label{tightchar}

It will be convenient for us to define $f(H):=2|V(H)|-|E(H)|$ for a graph $H$.

\begin{defn}
Let $\ell=1,2,3$. A simple graph $G$ is \emph{$(2,\ell)$-sparse} if $f(H) \geq \ell$ for all subgraphs $H$ of $G$ with at least one edge and is \emph{$(2,\ell)$-tight} if it is $(2,\ell)$-sparse and $f(G)=\ell$.
\end{defn}

We begin by recalling and formally defining the construction moves under consideration. Define the \emph{Henneberg $0$} move to be the addition of a vertex of degree $0$ or of degree $1$ to a graph. The inverse Henneberg $0$ move is the removal of a vertex of degree $0$ or degree $1$ from a graph.

The \emph{Henneberg $1$ move} \cite{Hen}, is the addition of a degree $2$ vertex to a graph. The inverse Henneberg $1$ move is the removal of a degree $2$ vertex from a graph.

The \emph{Henneberg $2$ move} \cite{Hen}, removes an edge $uv$ and adds a vertex $x$ and edges $xu,xv,xw$ for some vertex $w$. The inverse Henneberg $2$ move removes a degree $3$ vertex $x$ (and incident edges $xu,xv,xw$) and adds an edge $uv, uw$ or $vw$.

Let $G$ be $(2,1)$-sparse containing a copy of $K_4$. Write $G / K_4$ for the (possibly multi)graph formed by contracting this copy of $K_4$ to a vertex $v_*$. That is $G / K_4$ has vertex set $(V(G) \setminus V(K_4)) \cup \{v_{*}\}$ and edge set $(E(G) \setminus E(K_4)) \cup E_{*}$ where $E_{*}$ consists of the edges $vv_{*}$ associated with edges $vw$ where $v \in G / K_4$ and $w \in K_4$. We call this operation a \emph{$K_4$-to-vertex move}. The inverse move, he \emph{vertex-to-$K_4$ move} is illustrated in Figure \ref{fig:vtok4}.

The graph extension move mentioned in the introduction refers to the construction of $G$ from $G/H$ where $H$ is a proper induced $(2,2)$-tight subgraph of $G$. This move was used in \cite{NOP} and is similar to
vertex expansion moves used in graph theory, \cite{Die}.

\begin{center}
\begin{figure}[ht]
\centering
\includegraphics[width=5cm]{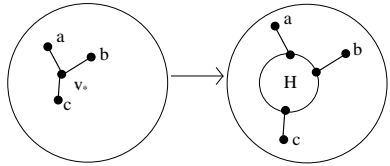}
\caption{With $H=K_4$, an example of the vertex-to-$K_4$ move and, with $H$ a proper induced $(2,2)$-tight subgraph of $G$, graph extension.}
\label{fig:vtok4}
\end{figure}
\end{center}

Let $G$ be a graph with an edge $uv$ such that the neighbours of $v$ are $a_1,\dots,a_n$. The \emph{edge-to-$K_3$ move}, see Figure \ref{Vertex splitting*}, (often referred to as vertex splitting in the literature, \cite{Whi6}) removes the edge $uv$ and the vertex $v$ and all the edges $va_i$, it replaces them with the vertices $v_1,v_2$ and edges $uv_1,uv_2,v_1v_2$, plus some bipartition of the remaining edges $v_1a_j$ and $v_2a_k$ (with one side possibly empty). The inverse move, called the \emph{$K_3$-to-edge move}, takes a copy of $K_3$ (with vertices $u,v_1,v_2$), removes the edges $uv_1,uv_2,v_1,v_2$, merges two vertices $v_1,v_2$ into a single vertex $v$ which is adjacent to all the vertices $v_1$ and $v_2$ were adjacent to and adds the edge $uv$.

\begin{center}
\begin{figure}[ht]
\centering
\includegraphics[width=6cm]{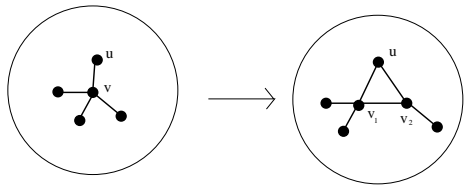}
\caption{The edge-to-$K_3$ move.}
\label{Vertex splitting*}
\end{figure}
\end{center}

Let $K$ and $H$ be $(2,1)$-tight graphs with vertices $u \in K$ and $v \in H$. We will say that $G$ is formed from $K$ and $H$ by an \emph{edge joining move}, see Figure \ref{Join} if $V(G)=V(K)\cup V(H)$ and $E(G)=E(K)\cup E(H) \cup \{uv\}$. Further, if there is a single edge $uv$ joining two $(2,1)$-tight subgraphs $G$ and $H$, then we will call the inverse move an \emph{edge separation move}.

\begin{center}
\begin{figure}[ht]
\centering
\includegraphics[width=8cm]{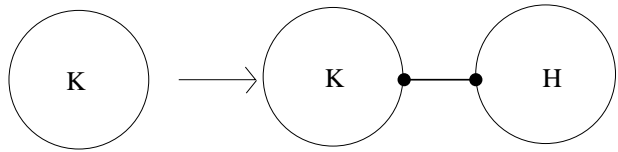}
\caption{The edge joining move.}
\label{Join}
\end{figure}
\end{center}

With respect to Theorem \ref{22refine} note that Figure \ref{edgetok3nec} illustrates the necessity of the $K_3$-to-edge move when we restrict graph contraction to the $K_4$-to-vertex move. 

We note that $(1) \Leftrightarrow (3)$ in Theorem \ref{21theorem} can be proven in an elementary way by showing that the construction operations preserve the spanning subgraph decomposition. More efficiently, these implications follow from matroidal results; the $(1,1)$-tight graphs are the bases of the cycle matroid and the $(1,0)$-tight graphs are the bases of the bicycle matroid. The union (on the same ground set of vertices) of a cycle matroid and a bicycle matroid (with empty intersection) give the results, see \cite{Fr&S}, \cite{G&W}, \cite{N-W}, \cite{Whi5}.

\begin{center}
\begin{figure}[ht]
\centering
\includegraphics[width=5cm]{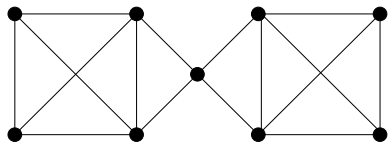}
\caption{A $(2,2)$-tight graph that requires the $K_3$-to-edge move when we restrict the graph contraction move to the $K_4$-to-vertex move.}
\label{edgetok3nec}
\end{figure}
\end{center}

\section{Construction Theory}
\label{sec2*}

In this section we consider $(1) \Rightarrow (2)$ in Theorems \ref{21theorem} and \ref{22refine}. That is, we consider whether an arbitrary $(2,2)$ or $(2,1)$-tight graph can be reduced by applying one of our short list of moves (relevant to each case) to produce a smaller $(2,2)$-tight or $(2,1)$-tight graph. 

We begin by showing that in a $(2,1)$-tight or $(2,2)$-tight graph, an inverse Henneberg 2 move is available unless all degree 3 vertices are in copies of $K_{4}$.

\begin{lem}\label{k4lem}
 Let $G$ be $(2,\ell)$-tight for $\ell=1,2$ with a vertex $v\in V(G)$ of degree $3$ with neighbours $v_{1},v_{2},v_{3}$ in $G$. Then either $v$ is contained in a copy of 
$K_{4}$ or $G'=(G\setminus v) \cup e$ (for $e=v_1v_2,v_2v_3$ or $v_3v_1$) is $(2,1)$-tight.
\end{lem}

\begin{proof}
With suitable labeling of vertices, we distinguish three cases corresponding to the possible edges among the neighbours of $v$.  Either

\begin{enumerate}
\item $v_1v_2,v_1v_3,v_2v_3 \in E$, 
\item $v_1v_2 \notin E,v_1v_3,v_2v_3 \in E$, or
\item $v_1v_2,v_2v_3 \notin E$.

\end{enumerate}

In case 1, $v,v_1,v_2,v_3$ induce a copy of $K_4$ in $G$.

\begin{center}
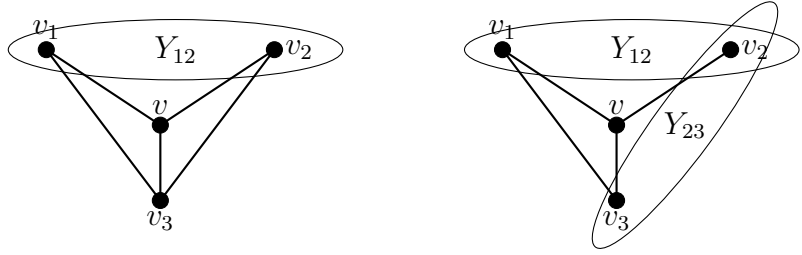
\begin{figure}[ht]
\centering
\begin{tikzpicture}
 \filldraw[black] 
(0,0) circle (3pt)node[anchor=north]{$v_3$}
(0,1) circle (3pt)node[anchor=south]{$v$}
(-1.5,2) circle (3pt)node[anchor=south]{$v_1$}
(1.5,2) circle (3pt)node[anchor=west]{$v_2$}

(6,0) circle (3pt)node[anchor=north]{$v_3$}
(6,1) circle (3pt)node[anchor=south]{$v$}
(4.5,2) circle (3pt)node[anchor=south]{$v_1$}
(7.5,2) circle (3pt)node[anchor=west]{$v_2$};
 \draw[black,thick]
  (0,0) -- (0,1) -- (-1.5,2) -- (0,0);
 \draw[black,thick]
  (0,1) -- (1.5,2) -- (0,0);
   \draw[black,thick]
  (6,0) -- (6,1) -- (4.5,2) -- (6,0);
   \draw[black,thick]
  (6,1) -- (7.5,2);
\draw[black] (0.2,2) ellipse (2.2 and 0.4)node{$Y_{12}$};
\draw[black] (6.2,2) ellipse (2.2 and 0.4)node{$Y_{12}$};
\draw[black, rotate around={54:(6.9,1)}] (6.9,1) ellipse (2 and 0.4)node{$Y_{23}$};
\end{tikzpicture}
\caption{The graph on the left illustrates case (2): if there was a subgraph $Y_{12}$ preventing the application on an inverse Henneberg 2 move on $v$ then the graph pictured would be over-counted. Similarly the graph on the right illustrates case (3): if there are subgraphs preventing the addition of $v_1v_2$ and $v_2v_3$ then the graph pictured would be over-counted. }
\label{fig:21h2proof}
\end{figure}
\end{center}

Figure \ref{fig:21h2proof} illustrates the proof in cases $2,3$. Define $Y_{12}$ to be a $(2,\ell)$-tight subgraph of $G$ containing $v_1,v_2$ but not $v_3,v$. Similarly define $Y_{13}$ and $Y_{23}$.

In case 2, $G'=(V\setminus v, (E\setminus \{vv_1,vv_2,vv_3\})\cup v_1v_2)$ is $(2,\ell)$-tight unless there exists a subgraph $Y_{12}$
of $G$. But then the addition of $v,v_3$ and their five incident edges to $Y_{12}$ gives a subgraph $Y$ of $G$ with $f(Y)=\ell-1$ which contradicts the fact that $G$ is $(2,\ell)$-tight.

In case 3, either $G'=(V\setminus v, (E\setminus \{vv_1,vv_2,vv_3\})\cup v_1v_2)$ or  
$G'=(V\setminus v, (E\setminus \{vv_1,vv_2,vv_3\})\cup v_2v_3)$ is $(2,\ell)$-tight unless there exists subgraphs $Y_{12}$ and $Y_{23}$
of $G$. Then 
\[f(Y_{12}\cup Y_{23})=f(Y_{12})+f(Y_{23})-f(Y_{12}\cap Y_{23}) \leq \ell+\ell-\ell=\ell \]
since $Y_{12}\cap Y_{23} \supseteq v_2$ and $Y_{12}\cap Y_{23}\subset G$. But then the addition of $v$ and its three incident edges to $ Y_{12}\cup Y_{23}$ gives a subgraph $Y$ of $G$ with $f(Y)=\ell-1$ which contradicts the fact that $G$ is $(2,\ell)$-tight.
\end{proof}

\begin{lem} \label{jco6}
Let $G=(V,E)$ be a $(2,\ell)$-tight graph for $\ell=1,2$. Then $G$ has either
an inverse Henneberg 1 move, an inverse Henneberg 2 move or at least $2l$ degree $3$ vertices, each of which is in a copy of $K_4$.
\end{lem}

\begin{proof}
$G$ has no degree $1$ vertices since this would imply that there is an edge $ab \in E(G)$ and 
$G=Y \cup ab$ with $b \notin V(Y)$ and $f(Y) = l-1$.

Assume $G$ has no inverse Henneberg 1 move. Then every vertex has degree at least three.

Label the vertices $1,\dots, |V|$ and let $d(i)$ denote the degree of vertex $i$. The summation over the degree of all vertices in $G$ gives $2|E|$. Hence the 
condition that $G$ is $(2,\ell)$-tight gives
\begin{equation}
\sum_{i=1}^{|V|}(4-d(i))=2l.
\label{degreecount}
\end{equation}

Since $d(i) \geq 3$ this implies $G$ has at least $2\ell$ degree $3$ vertices. By
Lemma \ref{k4lem} $G$ has an inverse Henneberg 2 move or each of these degree $3$ vertices is in a copy
of $K_4$.
\end{proof}

We will say that a $K_3$-to-edge or a $K_4$-to-vertex move is \emph{allowable} if it results in a graph
which is simple and $(2,\ell)$-tight.

The following lemma shows that a $K_4$-to-vertex move is allowable provided that the copy of $K_4$ does not have two vertices in a single copy of $K_3$.

We use the notation $K_n(v_1, \dots, v_n)$ for a subgraph of $G$ which is a copy of the complete
graph $K_n$ on the vertices $v_1, \dots, v_n$. 

\begin{lem} \label{jco7}
Let $G$ be a $(2,\ell)$-tight graph with $|V(G)|>4$ and let $G \rightarrow  G/K_4$ be a $K_4$-to-vertex move. Then 
$G/K_4$ is 
simple and $(2,\ell)$-tight unless there is a $K_3$ in $G$ with 
$|V(K_3 \cap K_4)|=2$.
\end{lem}

\begin{proof}
$G/K_4$ is simple unless there is a vertex $v \in V(G) \setminus V(K_4)$ and edges $a,b \in E(G)$ with 
$a,b \in V(K_4)$. In this case $|V(K_3(v,a,b) \cap K_4)|=2$.

$f(G/K_4)=f(G)$ so $G/K_4$ is $(2,\ell)$-tight unless there is a 
$Y' \subset G/K_4$ with 
$f(Y')<l$. There is a corresponding $Y \subset G$ such that $Y' = Y/K_4$. But then $f(Y)<l$ because $f(Y)=f(Y')$
which contradicts the $(2,\ell)$-sparsity of $G$.
\end{proof}

The following lemma describes when a $K_3$-to-edge move is allowable. Note that a $(2,\ell)$-tight graph containing no 
copy of $K_3$ admits an inverse Henneberg move by Lemmas \ref{jco6} and \ref{k4lem}.

\begin{lem} \label{jco1}
Let $G$ be a $(2,\ell)$-tight graph and $G \rightarrow  G'$ a $K_3$-to-edge move in which 
the vertices $a, b \in K_3(a,b,c)$ are the vertices in $G$ which are merged. Then $G'$ is 
simple unless there is a $K_3(a,b,d)$ in $G$ with $d \neq c$ and $G'$ is $(2,\ell)$-sparse unless there is a 
$Y \subset G$ with $ab \in E(Y)$, $c \not\in V(Y)$ and $Y$ is $(2,\ell)$-tight.
\end{lem}

\begin{proof}
$G'$ is simple provided there is no vertex $d$ different from $c$ and two edges $da$,
$db$. This gives the first condition.

$G'$ is $(2,\ell)$-sparse provided it has no
subgraph $Y'$ with $f(Y') < l$. $Y'$ is also a subgraph of $G$ unless it 
derives from a subgraph $Y \subset G$ with $ab \in E(Y)$ and $f(Y') < f(Y)$ only if 
$c \not\in Y$.
\end{proof}

There are three possible $K_3$-to-edge moves which can be applied to a copy of $K_3$ in $G$. If none of these results in a simple graph then there are three 
further copies of $K_3$ in $G$ and, if these are distinct, there are six further $K_3$-to-edge moves which might result in a simple graph. We will use this 
growth in the number of copies of $K_3$ to show that if $G$ contains a copy of $K_3$ then either $G$ has an allowable $K_3$-to-edge move
or every copy of $K_3$ is in a copy of $K_4$ (Lemma \ref{jco4} below). This $K_4$ gives an allowable $K_4$-to-vertex move unless it is 
adjacent to a copy of $K_3$ which, by this argument, must also be in another copy of $K_4$. This allows us to put a 
strong constraint on the possible graphs which contain a copy of $K_3$ but no allowable $K_3$-to-edge or $K_4$-to-vertex move (Lemma \ref{jco5} below). 

In order to keep track of the way in which copies of $K_3$ may share edges in a $(2,\ell)$-tight graph we first define a triangle sequence which is a set of 
nested subgraphs of $G$
and derive some of its properties. 

\begin{defn} \label{jco2}
Let $G$ be a simple graph. A triangle sequence in $G$ is a nested set of subgraphs
\[ M_3 \subset M_4 \subset \dots \subset M_i \dots \subset M_n \subseteq G\]
where $M_3$ is a copy of $K_3$, $E(M_i)$ and $V(M_i)$ are respectively the sets of edges and vertices of $M_i$, 
$|V(M_i)|=|V(M_{i-1})|+1$ and if $V(M_i) \setminus V(M_{i-1})=v_i$ then $E(M_i) \setminus E(M_{i-1})=v_ia_i,v_ib_i$ where 
$a_ib_i \in E(M_{i-1})$ and $a_ib_i$ is in exactly one copy of $K_3$ in $M_{i-1}$. We use $S(M_i)$ to denote the set of 
edges in $E(M_i)$ which are in exactly one copy of $K_3$ in $M_i$ (so $a_ib_i \in S(M_{i-1}))$.

%and $M_i$ is obtained from $M_{i-1}$ by the addition of a vertex 
%$v_i \in V(G) \setminus V(M_{i-i})$ and two edges $a_iv_i$ and $b_iv_i$ where $a_ib_i$ is an edge in $E(M_{i-1})$ which is in exactly one copy of $K_3 \in M_{i-1}$.
\end{defn}
We will often refer to a triangle sequence by the largest graph in the sequence. A maximal length triangle sequence is one which cannot 
be extended by a single vertex in $G$. We 
note that even for a maximal length triangle sequence with largest graph $M_n$ the graph $G$ may contain 
copies $K_3(a,b,c)$ which are not subgraphs of $M_n$ even though $ab \in E(M_n)$. This may occur if $c \in V(M_n)$ or if 
$c \notin V(M_n)$ and the edge $ab$ is in more than one copy of $K_3$ in $M_n$. Since $M_n$ is itself a graph we may form 
different 
triangle sequences within $M_n$ for example by starting with different copies of $K_3$ in $M_n$, see Figure \ref{johnfig}.

\begin{center}
\begin{figure}[ht]
\centering
\includegraphics[width=7cm]{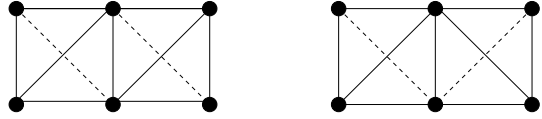}
\caption{Two possible maximal length triangle sequences in $K_4 \sqcup K_4$. In each case the graph shown is the final graph $M_6$ in the 
sequence. The dashed lines represent edges which are in $K_4 \sqcup K_4$ but not in $M_6$. The intermediate 
graphs in the sequence are obtained by starting with any copy of $K_3$ and sequentially adding one vertex and two 
edges from an adjacent copy of $K_3$}
\label{johnfig}
\end{figure}
\end{center}

\begin{lem} \label{jco3}
A triangle sequence in $G$ has the following properties.

\begin{enumerate}
\item $|V(M_i)|=i$ and $|E(M_i)| = 2i-3$.
\item The edges in $S(M_i)$ form a spanning cycle of $M_i$.
\item For every edge $ab \in E(M_i) \setminus S(M_i)$ the vertex pair $a,b$ separates $M_i$ (as a graph) with the property 
that if 
$aa_l,aa_r  \in S(M_i)$ then $a_l,a_r$ are in different separation components.
\item If $K_3(a,b,c)$ is any copy of $K_3$ in $M_n$ then there is a triangle sequence 
$M_3' \subset \dots \subset M_m'$ in $M_n$ such that $M_3'=K_3(a,b,c)$ and $M_m' = M_n$.
\end{enumerate}
\end{lem}

\begin{proof}
Property $(1)$ follows by induction since $|V(M_i)|=|V(M_{i-1})|+1$ and $|E(M_i)|=|E(M_{i-1})|+2$.

Property $(2)$ follows by induction. The edges of $S(M_3)$ form a spanning cycle of $M_3$. Assume 
property $(2)$  is true for $M_{i-1}$. Let $V(M_i)=V(M_{i-1}),v_i$ and let $E(M_i)=E(M_{i-1}),v_ia_i,v_ib_i$. The edge $a_ib_i \in S(M_{i-1})$ is in $K_3(a_i,b_i,v_i) \subset M_i$ 
in addition to a copy of $K_3$ in $M_{i-1}$ so is not in $S(M_i)$. The two edges $a_iv_i$ and 
$b_iv_i$ are both in $K_3(a_i,b_i,v_i)$ (and in no other copy of $K_3$ in $M_i$) so these are in 
$S(M_i)$. If the edges in $S(M_{i-1})$ form a spanning cycle $C_{i-1}$ of $M_{i-1}$ then the cycle
\[ C_i = (C_{i-1} \setminus a_ib_i) \cup a_iv_i \cup b_iv_i\]
forms a spanning cycle of $M_i$.

Property $(3)$ is also proved by induction. It is trivially true for $M_3$. Assume it is true for any $M_{i-1}$.
Let $V(M_i)=V(M_{i-1}),v_i$ and let $E(M_i)=E(M_{i-1}),v_ia_i,v_ib_i$. Every vertex pair which separates $M_{i-1}$ also 
separates $M_i$ with the same components because the vertices $a,b$ are adjacent and so are in the same component of any 
vertex separation of $M_{i-1}$. Putting vertex $v_i$ in this component gives a corresponding vertex separation of $M_i$. 
The edge $ab$ is the only edge which is in $E(M_i) \setminus S(M_i)$ and not in $E(M_{i-1}) \setminus S(M_{i-1})$. The 
vertex pair $a,b$ separates the vertex $v_i$ from the vertices $V(M_{i-1}) \setminus a,b$. The neighbours of $a$ in 
$S(M_i)$ are $v_i$ and a vertex $a_l \neq b \in V(M_{i-1}) \setminus a,b$ and these are separated by $a,b$.

For property $(4)$ we will show there is a triangle sequence in $M_n$ starting with $K_3(a,b,c)$ and 
terminating with $M_m'$ for
which $m=n$. This implies $|E(M_m')|=|E(M_n)|$ and that $M_m'=M_n$. Let $M_3',M_4' \dots M_m'$ be a
maximal length triangle sequence in $M_n$ starting with $K_3(a,b,c)$. Suppose for a contradiction that $m < n$. The edges in  
$S(M_m')$ form a spanning cycle 
of $M_m'$ and there is a edge $a_mb_m$ in $S(M_m')$ which is not in $S(M_n)$
(since $m<n$ and a cycle contains no proper subcycles). Since $a_mb_m$ is in $E(M_n)$ but not in $S(M_n)$ there is a vertex $v_m$ in $V(M_n) \setminus V(M_m')$ such that
there is $K_3(a_m,b_m,v_m)$ which in is in $M_n$ and not in $M_m'$. The edge $a_mb_m$ is therefore in a subgraph
$K_3(a_m,b_m,v_m)$ of $M_n$ but is not in $M_m'$. This implies that 
$v_m \in V(M_n) \setminus V(M_m')$ and $a_mv_m, b_mv_m \in E(M_n) \setminus E(M_m')$
with $a_mb_m \in S(M_m')$. This contradicts the maximality of the triangle sequence
in $M_n$.
\end{proof}

The following lemma uses a maximum length triangle sequence to show that if $G$ has a copy of $K_3$ but does not have a 
$K_3$-to-edge move then every edge in a copy of $K_3$ in $G$ is in at least two copies of $K_3$ in $G$.

\begin{lem} \label{jco4a}
Let $G$ be a $(2,\ell)$-tight graph for $l=1,2$ containing a copy of $K_3$. Then either
\begin{enumerate}
\item[$(i)$] there is a $K_3$-to-edge move in $G$ which gives a $(2,l)$-tight graph or
\item[$(ii)$] every edge in a copy of $K_3$ in $G$ is in at least two copies of $K_3$ in $G$.
\end{enumerate}
\end{lem}

\begin{proof}
Suppose that an edge $e=ab \in E(G)$ is in exactly one copy $K_3(a,b,c) \subset G$. By Lemma \ref{jco1}, the $K_3$-to-edge 
move which merges vertices $a$ and $b$ gives a simple graph $G^\prime$ and $G^\prime$ is $(2,\ell)$-tight unless $ab$ and 
$c$
have the following property (P): there is a
$(2,\ell)$-tight subgraph $Y \subset G$ with $ab \in E(Y)$ and $c \notin V(Y)$.

Suppose for a contradiction to the lemma that every edge in $G$ which is in exactly one copy of $K_3$ satisfies this 
property.

Let $M_3 \subset M_4 \dots \subset M_i \dots \subset M_n \subset G$ 
be a maximal length triangle sequence in $G$. Every edge in $E(M_n) \setminus S(M_n)$ is in two copies of $K_3$.

Suppose there is exactly one edge $ab$ in $S(M_n)$ which is in exactly one copy of $K_3$ in $G$ and therefore satisfies 
property (P) with corresponding subgraph $Y$. 
We will show by induction that 
$V(M_n) \cap V(Y) = \{a,b\}$ and $f(Y \cup M_n) = \ell$. Since $ab \in S(M_n)$ there is a vertex $c$ in $V(M_n)$ such that 
$K_3(a,b,c)$ is in $M_n$. 
By property $(4)$ of Lemma \ref{jco3} 
there is a triangle sequence $M_3' \subset \dots M_i' \dots \subset M_n$, starting with 
$M_3' = K_3(a,b,c)$ and ending with $M_n$.
$V(M_3') \cap V(Y) = \{a,b\}$ and $f(Y \cup M_3') = \ell$.
Assume for the induction that $V(M_{i-1}) \cap V(Y) = \{a,b\}$ and that $f(Y \cup M_{i-1}) = \ell$. 
Let $V(M_i) \setminus V(M_{i-1})=v_i$. If
$v_i \in V(Y)$ then
\[f(Y \cup v_ia_i \cup v_ib_i) = \ell-2\]
which contradicts the $(2,\ell)$-sparsity of $G$. Thus $V(M_i) \cap V(Y) = \{a,b\}$ and $f(Y \cup M_i) = \ell$. 

Every edge $cd$ in $S(M_n) \setminus ab$ is in a subgraph $K_3(c,d,v)$ of $G$ 
where $K_3(c,d,v)$ is not a subgraph of $M_n$.
Since $M_n$ is the largest graph in a maximal length triangle sequence we must have $v \in V(M_n)$ else $M_n$ could be 
extended to include $v$. But then 
$f(Y \cup M_n \cup cv)=\ell-1$ and since $Y \cup M_n \cup cv$ is a subgraph of $G$ this contradicts the 
$(2,\ell)$-sparsity of $G$.

Suppose there is more than one edge in $S(M_n)$ which is in exactly one copy of $K_3$ in $G$.
There are subgraphs $Y_1$ and $Y_2$ 
and edges $a_1b_1 \in Y_1 \cap M_n$ and $a_2b_2 \in Y_2 \cap M_n$. If the vertices 
$a_1,b_1,a_2,b_2$ are distinct then
\[f(Y_1  \cup Y_2 \cup M_n) \leq 2\ell-3\]
because $f(Y_1 \cup Y_2) \leq 2l$ and there are $n-4$ vertices and $2n-5$ edges in $M_n$ which are not in 
$Y_1 \cup Y_2$. If two of the vertices $a_1,b_1,a_2,b_2$ are the same then
\[f(Y_1  \cup Y_2 \cup M_n) = \ell-1\]
since $f(Y_1 \cup Y_2) = \ell$ and there are $n-3$ vertices and 
$2n-5$ edges in $M_n$ which are not in $Y_1 \cup Y_2$. In either case this contradicts the $(2,\ell)$-sparsity of $G$ for 
$l=1,2$.
\end{proof}

We say that an edge $ab \in E(G)$ is a $chord$ of $M_n$ if $a,b \in V(M_n)$ and $e \notin E(M_n)$. Let 
$[M_n]$ denote the graph induced in $G$ by $V(M_n)$. Then 
$E([M_n]) \setminus E(M_n)$ is the set of chords of $M_n$. The set $C$ defined in the next lemma is the set of edges in 
$S(M_n)$ which are in two or more copies of $K_3$ in $[M_n]$. We will show that when $M_n$ is the largest subgraph in a 
maximal length triangle sequence this is the same as a set of edges in $S(M_n)$ which are in two or more copies of $K_3$ in 
$G$. This lemma can then be used to limit the length of a triangle sequence because the number of chords of $M_n$ is 
limited to one for $\ell=2$ and to two for $\ell=1$ by the $(2,\ell)$-sparsity of $G$. 

We use the notation 
$\cup_{i=1}^m A_i$ to denote $A_1 \cup A_2,\dots,\cup A_m$ where $A_i$ are sets or graphs.

\begin{lem} \label{jco10}
Let $G$ be graph and let $M_n$ be a subgraph in a triangle sequence in $G$ with $n > 4$.
Let $e_1,\dots,e_m$ for $m > 0$ be chords of $M_n$, 
%let $H=E(M_n)\cup_{i=1}^m e_i$ and 
let $C_i=\{f \in S(M_n) : \exists g \in E([M_n])$ such that 
$K_3(e_i,f,g) \subset G\}$. Then $|C| \leq 3m$ where $C=\cup_{i=1}^m C_i$.
\end{lem}

\begin{proof}
Assume for induction that the lemma is true for all possible choices of $m-1$ chords of $M_n$ and suppose that
$e_1,\dots,e_m$ are a set of $m$ chords of $M_n$. 
 
Suppose the chords $e_1,\dots,e_m$ determine a graph with $t$ distinct vertices and $c$ connected components. Then 
$t \leq m+c$. Since the edges in $S(M_n)$ form a spanning cycle of $M_n$ each vertex of a 
chord $e_i$ is incident to two edges in $S(M_n)$. This implies $|C| \leq 2(m+c)$ which implies $|C| \leq 3m$ unless $c > m/2$. We may assume therefore that 
there is at least one component with exactly one 
edge which we label as the edge $e_m$ where $e_m$ has no vertices in common with $e_i,i=1,\dots,m-1$.

Let $e_m=ab$ with $a,b \in V(M_n)$. Each of the vertices $a,b$ is 
incident to exactly two edges in $S(M_n)$ which we label $aa_l,aa_r,bb_l,bb_r \in S(M_n)$. These 
edges are all distinct because $ab \notin S(M_n)$ . Since the edges in $S(M_n)$ form a cycle we may label the vertices so that there is a 
(possibly trivial) path $P(a_l,b_l) \in S(M_n)$ which connects $a_l,b_l$ and avoids $a,b,a_r,b_r$ and then 
$a_r \neq b_l$ and $a_l \neq b_r$, see Figure \ref{lemma3.8}. We may also label so that $a_r \neq b_r$ since if $a_l=b_l$ and 
$a_r=b_r$ the edges $aa_l,a_lb,ba_r,a_ra$ form a 4-cycle in $S(M_n)$ which contradicts $n > 4$.

\begin{center}
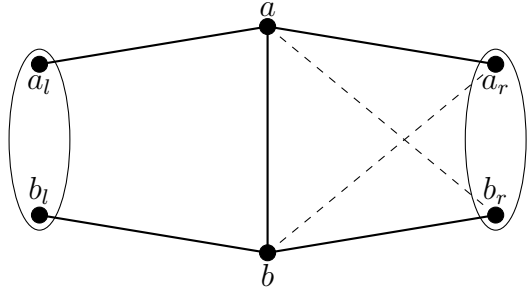
\begin{figure}[ht]
\centering
\begin{tikzpicture}
 \filldraw[black] 
(0,0) circle (3pt)node[anchor=south]{$b_l$}
(0,2) circle (3pt)node[anchor=north]{$a_l$}
(3,-.5) circle (3pt)node[anchor=north]{$b$}
(3,2.5) circle (3pt)node[anchor=south]{$a$}
(6,0) circle (3pt)node[anchor=south]{$b_r$}
(6,2) circle (3pt)node[anchor=north]{$a_r$};
 \draw[black,thick]
  (0,0) -- (3,-.5) -- (3,2.5) -- (0,2);
\draw[black,thick] (3,2.5) -- (6,2);
\draw[black,thick] (6,0) -- (3,-.5);
\draw[black] (0,1) ellipse (0.4 and 1.2);
\draw[black] (6,1) ellipse (0.4 and 1.2);
\draw[black,dashed] (3,2.5) -- (6,0);
\draw[black,dashed] (3,-.5) -- (6,2);
\end{tikzpicture}
\caption{A chord $ab$ of $M_n$ with $ab\in V(M_n)$ and adjacent to edges $aa_l, aa_r,bb_l, bb_r\in S(M_n)$. Edges $ab_r$ and $ba_r$ cannot both be in $E(M_n)$ because the vertex pair $b,a_r$ must then separate $M_n$.}
\label{lemma3.8}
\end{figure}
\end{center}

Any edge $f \in S(M_n)$ which is in a 3-cycle with $ab$ has a vertex in common with $ab$. Given edges $e,f$ there is at most one 
3-cycle in $G$ which includes $e,f$. Thus we have shown $|C_m| \leq 4$. Furthermore, if $|C_m|=4$ the vertex triples $a,b,a_r$ and $a,b,b_r$ must 
both induce 3-cycles in $[M_n]$. This implies $ba_r, ab_r \in E(M_n)$ because edges in $e_1,\dots,e_{m-1}$ have no vertices in common with $ab$. This contradicts 
Lemma \ref{jco3} part (3) for the vertex pair $b,a_r$ because 
the neighbours of $b$ in $S(M_n)$ are $b_l,b_r$ and there would be a path $b_ra,aa_l,P(a_l,b_l)$ which connects $b_r$ and $b_l$ and excludes $b,a_r$. 
Thus $|C_m| \leq 3$ which combines with the induction hypothesis $\cup_{i=1}^{m-1} C_i \leq 3(m-1)$ to give $\cup_{i=1}^m C_i \leq 3m$ .
\end{proof}

\begin{lem} \label{jco4}
Let $G$ be a $(2,\ell)$-tight graph for $\ell=1,2$ with the property that 
every edge $ab$ in a $K_3(a,b,c) \subset G$ is in at least two copies of $K_3$ in $G$.
Then every copy of $K_3$ in $G$ is in a copy of $K_4$.
\end{lem}

\begin{proof}
We will show first that every maximal length triangle sequence in $G$ with largest graph $M_n$ satisfies $n \leq 6$ for $\ell=1$ and $n \leq 4$ 
for $\ell=2$.

Since every edge $ab \in S(M_n)$ is in exactly one copy of $K_3$ in $M_n$ there is a vertex $c \in V(G)$ such that $K_3(a,b,c) \subset G$ and 
$K_3(a,b,c) \not\subset M_n$. This implies that $c \in V(M_n)$ because otherwise the triangle sequence can be extended with vertex $c$. Since 
$K_3(a,b,c) \not\subset M_n$ either $ac$ or $bc$ is a chord of $M_n$. Every edge in $S(M_n)$ must therefore be in the set $C$ defined in Lemma 
\ref{jco10} and if $n>4$ by Lemma \ref{jco10} we have $n =|C| \leq 3m$ where $m$ is the number of chords of $M_n$ in $G$. 
$f(M_n \cup_{i=1}^m e_i)=3-m$ because $f(M_n)=3$ so 
$f(M_n \cup_{i=1}^m e_i) \ge \ell$ implies $m \leq 3-\ell$ and $n \leq 3(3-\ell)$. These imply $n \leq 4$ for $\ell=2$ and 
$n \leq 6$ for $\ell=1$.
 
For $n=4$ there is a unique largest graph $M_4$ and a unique edge from $E([M_n]) \setminus E(M_n)$ which can be added to 
the graph $M_4$ so that every edge of $S(M_4)$ is in two copies of $K_3$. This creates a copy of $K_4$.

An analysis of the subgraphs induced by the vertices of maximal length triangle sequences $M_n$ with  $n \leq 6$  and with 
the 
property that every 
edge in $E(M_n)$ is in two copies of $K_3$ in $G$ shows that for $l=1$, $[M_5]=K_5\setminus e$ or
$[M_6]=K_4 \sqcup K_4$. Since every $K_3$ is in a maximal length triangle sequence and every $K_3$ in 
$K_4$, $K_5 \setminus e$ or $K_4 \cup K_4$ is in a copy of $K_4$ the proof is complete.
\end{proof}

\begin{lem} \label{jco5}
Let $G$ be a $(2,\ell)$-tight graph for $l=1,2$ which contains a copy of $K_3$. Then either $G=K_4$, $G$ has an 
allowable $K_3$-to-edge move, an allowable $K_4$-to-vertex move or every copy of $K_3$ is in a copy
of $K_4 \sqcup K_4$ or $K_5 \setminus e$.
\end{lem}

\begin{proof}
Let the copy of $K_3$ be $K_3(a,b,c)$ and assume $G$ has no allowable $K_3$-to-edge move or 
$K_4$-to-vertex move. By Lemmas \ref{jco4a} and \ref{jco4} $K_3(a,b,c)$ is in a $K_4(a,b,c,d)$. Since this does not give an allowable $K_4$-to-vertex move, by Lemma \ref{jco7} there is a $K_3(c,d,e)$ (say) with $a,b,c,d,e$ all distinct and again by
Lemma \ref{jco4} there is a $K_4(c,d,e,g)$. If $a,b,c,d,e,g$ are distinct then $K_3(a,b,c)$ is in a copy of $K_4 \sqcup K_4$ and if $g=a$ or $b$ then $K_3(a,b,c)$ is in a copy of $K_5 \setminus e$.
\end{proof}

We combine the lemmas in this section to show that all suitable $(2,\ell)$-tight graphs
have an allowable reduction move.

\begin{lem} \label{jco8}
Let $G$ be $(2,2)$-tight. Then $G=K_4$ or $G$ has an inverse Henneberg $1$ move,
an inverse Henneberg $2$ move, an allowable $K_3$-to-edge move or an allowable $K_4$-to-vertex move.
\end{lem}

\begin{proof}
Assume $G$ has no inverse Henneberg $1$ move and no inverse Henneberg $2$. By Lemma \ref{jco6} $G$ has a
copy of $K_4$ and thus a copy of $K_3$. The proof is completed by Lemma \ref{jco5} since neither
$K_4 \sqcup K_4$ nor $K_5 \setminus e$ is $(2,2)$-sparse.
\end{proof}

\begin{lem} \label{jco9}
Let $G$ be $(2,1)$-tight. Then $G=K_4 \sqcup K_4$ or $G=K_5 \setminus e$ or $G$ has an inverse 
Henneberg $1$ move, an inverse Henneberg $2$ move, an allowable $K_3$-to-edge move, an allowable $K_4$-to-vertex move or an edge separation move.
\end{lem}

\begin{proof}
Assume $G$ has no inverse Henneberg 1 move, no inverse Henneberg 2 move, no allowable 
$K_3$-to-edge move and no allowable $K_4$-to-vertex move. By Lemma \ref{jco6}
each of the degree-3 vertices in $G$ is in a copy of $K_4$ and thus in a copy of $K_3$. By Lemma 
\ref{jco5} each of these $K_3$ is in a copy of $K_4 \sqcup K_4$ or $K_5 \setminus e$.

Let $Y=\{Y_1, \dots, Y_n\}$ be the set of subgraphs of $G$ which are each copies of 
$K_4 \sqcup K_4$ or $K_5 \setminus e$.

The subgraphs $Y_i \in Y$ are vertex disjoint since
\[f(Y_i \cup Y_j)= f(Y_i)+f(Y_j)-f(Y_i \cap Y_j) =2-f(Y_i \cap Y_j)\]
and $(2,1)$-sparsity requires $f(Y_i \cap Y_j) \leq 1$. Every proper subgraph $X$ of $K_4 \sqcup K_4$
or $K_5 \setminus e$ has $f(X) \geq 2$ so this requires $Y_i$ and $Y_j$ to be vertex disjoint.

Let $V_0$ and $E_0$ be the sets of vertices and edges in $G$ which are in none of the $Y_i \in Y$.
Then
\[f(G)=\sum_{i=1}^nf(Y_i)+2|V_0|-|E_0|\]
so $|E_0|=2|V_0|+n-1$. Each of the vertices in $V_0$
is incident to at least 4 edges in $E_0$. If each $Y_i$ is incident to at least 2 edges in $E_0$
%there are at least $4|V_0|+2n$ edge/vertex incidences in $E$ which implies $|E_0| \geq 2|V_0|+n$ for
then $|E_0| \geq (4|V_0|+2n)/2$ for a contradiction.

At least one of the $Y_i$ is incident to at most one edge in $E_0$. If this $Y_i$ is incident to
no edges in $E_0$ then $G=K_4 \sqcup K_4$ or $G=K_5 \setminus e$ since $G$ is connected. Otherwise 
$Y_i$ is incident to one edge $e \in E_0$ and $e$ provides an edge separation move.
\end{proof}

Using the above lemmas we reach the stated goal of this section.

\begin{proof}[Proof of $(1)\Rightarrow(2)$ in Theorem \ref{21theorem} or Theorem \ref{22refine}]
By induction using Lemma \ref{jco8} or Lemma \ref{jco9}.
\end{proof}

\section{Further Work}
\label{21apps}

We expect to be able to use Theorem \ref{21theorem} to prove the following conjecture discussed in the introduction.

\begin{con}\label{tjcon}
Let $\M$ be a cone, a torus, a union of concentric cones or a union of concentric tori and let $p$ be generic. Then $(G,p)$ is generically minimally rigid on $\M$ if and only if $G=K_{2},K_3, K_4$ or $G$ is $(2,1)$-tight.
\end{con}

It would also be interesting to consider surfaces that do not admit any rigid-body motions. For such surfaces there are immediate additional problems. 
For example Equation \ref{degreecount} with $\ell=0$ shows that the minimum degree in a $(2,0)$-tight graph may be $4$ so additional Henneberg type
operations are required.
This actually provides additional motivation for studying these graphs since the obvious choices to take are $X$ and $V$-replacement 
as studied by Tay and Whiteley \cite{T&W} in the $3$-dimensional setting. Indeed they conjecture that these operations (with additional 
conditions for $V$-replacement)
preserve rigidity in $3$-dimensions. 

It is also interesting to note that the $d$-dimensional version of the edge-to-$K_3$ move, known in the literature as vertex 
splitting \cite{Whi6}, is one of a very short list of operations known to preserve rigidity in arbitrary dimension. Nevertheless there is no
conjectured inductive construction, even in $3$-dimensions, that makes use of this. We hope that our methods for dealing with the
edge-to-$K_3$ move for $(2,\ell)$-tight graphs may be useful in finding such a construction.

There are more exotic settings in which the class of $(2,1)$-tight graphs are the appropriate combinatorial tool needed to classify 
generic minimal rigidity. For example we could take $\M$ to be two parallel (but not concentric) cylinders. Here there is only one 
rigid-body  motion of $\M$ in $\bR^3$, or we may take $\N$ to be a cylinder coaxial to a cone. Again there is only one rigid-body motion
(this time a rotation about the central axis). In such reducible settings there is a little more work to do to in considering which surface each framework point lies on. This extra requirement is particularly evident for 
$\N$, but in either case a $(2,1)$-tight subgraph realised purely on one cylinder would be overconstrained.

A similar but deeper topic is the problem of when a framework realisation is unique (this is the topic of global rigidity, see for 
example \cite{J&J} and \cite{Con2}). To characterise the global rigidity of frameworks supported on an algebraic surface one of the key steps is to analyse the circuits of the rigidity matroid $\R_{\M}$ (this is the linear matroid induced 
by the linear independence of the rows of the surface rigidity matrix). Since the independent sets in $\R_{\M}$, for $\M$ a cylinder, 
may be identified 
with the $(2,2)$-tight graphs (\cite[Theorem $5.4$]{NOP}), the circuits may be identified with a sub-class of the 
$(2,1)$-tight graphs. Such a recursive construction is given in \cite{Nix} and finding a similar construction for circuits in the $(2,1)$-tight
matroid is open.

\end{document}